\documentclass[12pt,a4paper]{article}

\usepackage[british]{babel}

\usepackage[a4paper,top=2cm,bottom=2cm,left=2.5cm,right=2.5cm,marginparwidth=1.75cm]{geometry}

\usepackage[style=numeric, natbib=true]{biblatex} 
\usepackage{amsmath}
\usepackage{graphicx}
\usepackage[colorlinks=true, allcolors=blue]{hyperref}
\usepackage{hyperref}
\usepackage[title]{appendix}
\usepackage{mathrsfs}
\usepackage{amsfonts}
\usepackage{booktabs} 
\usepackage{caption}  
\usepackage{threeparttable} 
\usepackage{algorithm}
\usepackage{algorithmicx}
\usepackage{algpseudocode}
\usepackage{listings}
\usepackage{enumitem}
\usepackage{chngcntr}
\usepackage{booktabs}
\usepackage{lipsum}
\usepackage{subcaption}
\usepackage{authblk}
\usepackage[T1]{fontenc}    
\usepackage{csquotes}       
\usepackage{diagbox}
\usepackage{comment}

\usepackage{indentfirst} 
\usepackage{algorithm}
\usepackage{algorithmicx}
\usepackage{algpseudocode}
\usepackage{amsthm} 
\newtheorem{theorem}{Theorem} 
\usepackage{amsmath,amssymb}
\usepackage{newtxtext,newtxmath} 

\DeclareMathOperator{\rank}{rank}

\usepackage{setspace}
\usepackage{titlesec}
\titleformat{\section} 
  {\normalfont\Large\bfseries}{\thesection.}{1em}{}
  
\usepackage{lineno} 

\rightlinenumbers 

\usepackage{float}   
\usepackage{caption} 


\title{Block Coordinate Descent Network Simplex Methods for Optimal Transport}

\author{%
  Lingrui Li\thanks{Graduate School of Informatics, Kyoto University. 
    E-mail: li.lingrui.82y@st.kyoto-u.ac.jp.}%
  \and
  Nobuo Yamashita\thanks{Graduate School of Informatics, Kyoto University.}%
}

\date{}  

\begin{document}
\maketitle

\begin{abstract}

We propose the Block Coordinate Descent Network Simplex (BCDNS) method for solving large-scale Optimal Transport (OT) problems. BCDNS integrates the Network Simplex (NS) algorithm with a block coordinate descent (BCD) strategy, decomposing the full problem into smaller subproblems per iteration and reusing basis variables to ensure feasibility. We prove that BCDNS terminates in a finite number of iterations with an exact optimal solution, and we characterize its per-iteration complexity as $O(sN)$, where $s\in(0,1)$ is a user-defined parameter and $N$ is the total number of variables. Numerical experiments show that BCDNS attains the same exact objective value as the classical NS method while substantially reducing runtime on large instances, with speed-ups of up to tens of times over NS. In addition, comparisons with the stabilized Sinkhorn algorithm under different regularization levels illustrate that BCDNS converges to the unregularized OT optimum, and its runtime is competitive with Sinkhorn when high accuracy is required.

\end{abstract}



\section{Introduction}

The Optimal Transport (OT) problem is to find the most efficient way to
transport resources between a finite number of locations while minimizing transportation costs. It is rooted in mathematical optimization and has a rich history dating back to the 18th century. Gaspard Monge~\cite{monge1781deblais} first proposed it  in 1781. Then in the mid-20th century, Leonid Kantorovich~\cite{kantorovich1942translocation} extended Monge's framework to encompass continuous distributions, reformulating it as a linear programming model and significantly broadening its theoretical foundation and practical applications. In this paper, we focus specifically on solving the discrete OT problem.

The OT problem has found applications in various disciplines. In machine learning, the OT problem has been used for tasks such as domain adaptation~\cite{courty2016optimal} and generative modeling~\cite{arjovsky2017wasserstein}. In economics, the OT problem has provided insights into resource allocation, market equilibrium, and pricing mechanisms~\cite{galichon2016optimal}. In biology, the OT problem has been applied to analyze genetic expression patterns~\cite{schiebinger2019optimal}. In geophysics and signal processing, the OT problem has been used for seismic waveform comparison and inversion~\cite{liao2022fast}.

Several methods have been developed to solve the OT problem. Each method has different strengths and limitations. 
The Hungarian method is a classical algorithm for the linear assignment problem~\cite{kuhn1955hungarian},
which can be viewed as a special case of OT where the transport plan is restricted to a one-to-one matching.
It produces an exact optimum in $O(n^3)$ time for an $n\times n$ assignment instance.
However, for general OT, the Hungarian method is no longer applicable without
expanding the problem size, which becomes prohibitive for large-scale problems.
The Sinkhorn algorithm is a popular method for the large-scale OT problem due to its computational efficiency~\cite{cuturi2013sinkhorn}. It solves the OT problems with regularization of entropy, which is an approximation of the OT problem. Therefore it provides approximate solutions of the OT problem, and it may not be suitable for applications that require exact solutions. In recent years, many improvements based on the Sinkhorn algorithm have been proposed, such as the Greenkhorn algorithm~\cite{altschuler2017near} and the Adaptive Primal–Dual Accelerated Gradient Descent algorithm~\cite{dvurechensky2018computational}. However, these methods still operate within the Sinkhorn framework and do not overcome the limitation of producing only approximate OT solutions.
In contrast, the Network Simplex (NS) method~\cite{grigoriadis1986efficient}, a specialized variant of the simplex method, is well known to yield exact solutions. However, the NS method takes \(O(N)\) time in each iteration, which leads to slow computational performance when applied to large-scale problems.
Yue Xie et al.~\cite{xie2023randomized} proposed a block coordinate descent (BCD) approach for solving large-scale OT problems.
At each iteration, the method constructs a restricted subproblem by selecting a subset of variables according to a prescribed sampling rule and then solves the subproblem while keeping the remaining variables fixed.
By operating on a much smaller working set, the BCD approach can substantially reduce the computational burden per-iteration and, in practice, also reduces working memory compared to methods that process all $N$ variables at once.
However, since the BCD method proposed by Yue Xie et al. adapts the random variable selection, it is not guaranteed to find an optimal solution within a finite number of iterations. Moreover, it sometimes takes a large number of iterations to get a reasonable approximate solution.

The motivation behind this research lies in the growing demand for efficient methods capable of solving large-scale OT problems with highly accurate solutions. As data sizes continue to grow, we recognize the increasing importance of solving large-scale OT problems with both speed and accuracy. To address these challenges, our aim is to develop a method that enables fast computation of exact solutions for large-scale OT problems.

To achieve the goal of efficiently solving large-scale OT problems, we introduce the idea of the BCD method. The BCD method divides a large-scale OT problem into smaller subproblems, allowing a reduction in computational complexity. Although the BCD method is generally considered slow when exact solutions are required, we observe that the optimal solutions of the OT problem are typically sparse, which significantly improves the efficiency of the BCD method. By employing the NS method to solve the subproblems, we guarantee that the final solution is also exact.

In this paper, we develop the \textbf{Block Coordinate Descent Network Simplex (BCDNS)} method for solving large-scale OT problems with exact optimality.
Compared to existing BCD method~\cite{xie2023randomized} that randomly selects variables based on matrix properties, we introduced a basis variable succession strategy to reuse basis information across iterations, which ensures monotone improvement and finite termination.
Furthermore, we propose a grouped variable selection mechanism to reduce unnecessary reduced-cost evaluations.
Theoretical analysis establishes convergence to the global optimum in finitely many iterations,
and numerical experiments demonstrate that the proposed method significantly outperforms the classical NS method and is competitive with high-precision
Sinkhorn algorithms in terms of runtime when high accuracy is required.

The remainder of this paper is organized as follows.
Section 2 formulates the discrete OT problem and introduces the notation used throughout the paper.
Section 3 reviews representative approaches for OT, including the NS method, the BCD methods, and the Sinkhorn algorithm.
Section 4 presents the proposed BCDNS framework, together with concrete block selection strategies and theoretical properties.
Section 5 reports the numerical experiments, including the test settings and the empirical results.
Finally, Section 6 concludes the paper and discusses future directions.

\section{Discrete Optimal Transport}

We consider a discrete OT problem~\cite{villani2009optimal} between \( n \) supply locations and \( m \) demand locations. Let \( \{1, 2, \dots, n\} \) denote the index set of supply points, and \( \{1, 2, \dots, m\} \) denote that of demand points.

Each supply point \( i \in \{1, \dots, n\} \) provides a non-negative amount of mass \( p_i \geq 0 \), and each demand point \( j \in \{1, \dots, m\} \) requires an amount \( q_j \geq 0 \). The total supply and demand are assumed to be balanced:
\[
\sum_{i=1}^n p_i = \sum_{j=1}^m q_j.
\]

Let \( C = \{c_{ij}\} \in \mathbb{R}^{n \times m} \) be the cost matrix, where each entry \( c_{ij} \geq 0 \) denotes the cost of transporting one unit of mass from supply point \( i \) to demand point \( j \). A transport plan is represented by a matrix \( X = \{x_{ij}\} \in \mathbb{R}^{n \times m} \), where \( x_{ij} \) is the amount of mass transported from \( i \) to \( j \).

The discrete OT problem can then be formulated as the following linear program:
\begin{equation}
\label{eq:primal}
\begin{alignedat}{2}
    & \min_{X \in \mathbb{R}_{\geq 0}^{n \times m}} \quad & & \sum_{i=1}^n \sum_{j=1}^m c_{ij} x_{ij} \\
    & \text{subject to} \quad & & \sum_{j=1}^m x_{ij} = p_i, \quad \forall i = 1, \dots, n, \\
    &                         & & \sum_{i=1}^n x_{ij} = q_j, \quad \forall j = 1, \dots, m, \\
    &                         & & x_{ij} \geq 0, \quad \forall i = 1, \dots, n,\; j = 1, \dots, m.
\end{alignedat}
\end{equation}

To express this in standard matrix-vector form, we define the following notation. Let \( N = nm \) be the total number of transport variables, and let \( M = n + m \) be the number of equality constraints. We vectorize the matrix \( X \) into a column vector \( \mathbf{x} \in \mathbb{R}^{N} \) by stacking its columns, i.e.,
\[
\mathbf{x} = \operatorname{vec}(X) = [x_{11}, x_{21}, \dots, x_{n1}, x_{12}, \dots, x_{nm}]^\top.
\]

We also define:
\begin{itemize}
    \item the cost vector \( \mathbf{c} = \operatorname{vec}(C) \in \mathbb{R}^{N} \),
    \item the supply vector \( \mathbf{p} = [p_1, \dots, p_n]^\top \in \mathbb{R}^n \),
    \item the demand vector \( \mathbf{q} = [q_1, \dots, q_m]^\top \in \mathbb{R}^m \),
    \item the constraint vector \( \mathbf{b} = [\mathbf{p}^\top, \, -\mathbf{q}^\top]^\top \in \mathbb{R}^{M} \).
    \item The constraint matrix 
    \[
    A =
    \begin{bmatrix}
    I_n & I_n & \cdots & I_n \\
    -1_n^\top & 0 & \cdots & 0 \\
    0 & -1_n^\top & \cdots & 0 \\
    \vdots & \vdots & \ddots & \vdots \\
    0 & 0 & \cdots & -1_n^\top 
    \end{bmatrix}\in \mathbb{R}^{M \times N} ,
    \]
    where \( I_n \) denotes the identity matrix \( n \times n \), and \( 1_n \) is a column vector of dimensions \( n \).
\end{itemize}

With this notation, we can rewrite the OT problem as follows:
\begin{equation}
\label{eq:OT_vector}
\begin{alignedat}{2}
    & \min_{\mathbf{x} \in \mathbb{R}^{N}} \quad & & \mathbf{c}^\top \mathbf{x} \\
    & \text{subject to} \quad & & A \mathbf{x} = \mathbf{b},\\
    &&&    \mathbf{x}  \geq 0.
\end{alignedat}
\end{equation}

\section{Existing Methods for Optimal Transport}

In this section, we review several representative algorithms for solving the discrete OT problem. These include the classical \emph{NS} method, the \emph{Sinkhorn algorithm} based on entropy regularization, which is widely used for large-scale problems, and the recently proposed \emph{BCD} method.

\subsection{Network Simplex Method for Optimal Transport}

It is well known that the node--arc incidence matrix of a connected graph has rank
$\lvert V\rvert-1$ \cite{ahuja1993networkflows}.
Since the OT network is the complete bipartite graph with $M=n+m$ nodes, we have
$\rank(A)=M-1$. This allows us to partition the set of variable indices \( \mathcal{N}=\{1, \dots, N\} \) into two subsets: a set of basis indices \( \mathcal{B} \subseteq \mathcal{N} \) with \( | \mathcal{B}| = M - 1 \), and its complement \( \mathcal{N} \setminus \mathcal{B} \). The columns of \( A \) corresponding to \( \mathcal{B} \), denoted \( A_B \in \mathbb{R}^{M \times (M-1)} \), form a full-rank submatrix. 
The basic variables $\mathbf{x}_B$ are determined by the equality constraints
\[
A_B \mathbf{x}_B = \mathbf{b},
\]
while the nonbasic variables satisfy $\mathbf{x}_N=\mathbf{0}$.
A basic feasible solution (BFS) additionally requires $\mathbf{x}_B\ge \mathbf{0}$.
To assess the optimality of a BFS, we examine the reduced costs derived from dual variables(node potentials).
Let $\boldsymbol{\pi}\in\mathbb{R}^{M}$ denote a vector of node potentials.
For a given basis $B$ (which corresponds to a spanning tree),
we choose $\boldsymbol{\pi}$ to satisfy the \emph{dual optimality equations}
\begin{equation}\label{eq:tree_potential}
A_B^\top \boldsymbol{\pi} = \mathbf{c}_B .
\end{equation}
Since $A_B$ has rank $M-1$, the solution of~\eqref{eq:tree_potential} is unique up to an additive constant;
we fix this degree of freedom by setting the potential of an arbitrary reference node to zero.
In practice, $\boldsymbol{\pi}$ can be computed in $O(M)$ time by traversing the spanning tree.

Given $\boldsymbol{\pi}$, the reduced cost of any variable $\ell$ is
\begin{equation}\label{eq:reduced_cost}
\bar c_\ell \;=\; c_\ell - A_\ell^\top \boldsymbol{\pi},
\end{equation}
where $A_\ell$ denotes the $\ell$-th column of $A$.
Equivalently, in vector form $\bar{\mathbf{c}} = \mathbf{c} - A^\top \boldsymbol{\pi}$.
Due to the structure of the OT problem, each column \( A_{ij} \) has exactly one \( +1 \) and one \( -1 \), so the reduced cost simplifies to the closed-form expression
\[
r_{ij} = c_{ij} - (\pi_i - \pi_{n+j}).
\]
A BFS is optimal if and only if $\bar c_\ell \ge 0$ for all nonbasic indices $\ell\in N$, i.e., \( r_{ij} \ge 0 \) for all \( i, j \). 

To perform a pivot, we add the index \( (i^*, j^*) \) to the basis \( B \), inducing a unique cycle in the associated flow graph. 
Then choose the step size $\Delta$ as the largest admissible increase along the fundamental cycle, i.e.,
\[
\Delta \;=\; \min\{\, x_e : e \in \mathcal{C}^- \,\},
\]
where $\mathcal{C}^-$ denotes the set of cycle edges on which the flow is decreased. This guarantees that the updated flow remains feasible and nonnegative.
Subsequently, update the flow: \( x_{i^* j^*} \gets \Delta \), and for each edge \( (u,v) \) along the cycle, the value of \( x_{uv} \) increases or decreases by \( \Delta \) depending on the direction of the cycle. Finally, the variable that reaches zero and removes it from the basis, completing the pivot step~\cite{vanderbei2001linear}.
Each iteration of the NS method involves computing the dual potential vector \( \boldsymbol{\pi} \) in \( O(M) \) time, evaluating all reduced costs in \( O(N) \), and performing pivoting operations in \( O(M) \).

\subsection{Block Coordinate Descent Method for Optimal Transport}

In this subsection, we introduce the BCD method for the OT problem in the LP form~\eqref{eq:OT_vector}.

The BCD method iteratively selects a subset of coordinates \(\mathcal{H} \subseteq \mathcal{N} = \{1, 2, \dots, N\}\). For any feasible solution \( \mathbf{x} \), we define a subproblem of the BCD method

\begin{equation}
D(\mathbf{x} ; \mathcal{H}) \triangleq \arg\min _{\mathbf{d} \in \mathbb{R}^N}\left\{\mathbf{c}^T (\mathbf{x} + \mathbf{d}) \mid \mathbf{x}+\mathbf{d} \geq 0, \;A \mathbf{d}=0,\; d_i=0,\; \forall i \in \mathcal{N} \backslash \mathcal{H}\right\}.
\label{eq:DxH}
\end{equation}

\begin{equation}
q(\mathbf{x} ; \mathcal{H}) \triangleq \min _{\mathbf{d} \in \mathbb{R}^N}\left\{\mathbf{c}^T (\mathbf{x} + \mathbf{d}) \mid \mathbf{x}+\mathbf{d} \geq 0, \;A \mathbf{d}=0,\; d_i=0,\; \forall i \in \mathcal{N} \backslash \mathcal{H}\right\}.
\label{eq:qH}
\end{equation}

Here, \( D(\mathbf{x}; \mathcal{H}) \) represents the set of optimal solutions for the subproblem with \( \mathcal{H} \), and \( q(\mathbf{x}; \mathcal{H}) \) denotes the optimal objective value corresponding to \( d \in D(\mathbf{x}; \mathcal{H}) \)~\cite{tseng2009block}.

The iterative BCD algorithm proceeds as follows:
\begin{itemize}
    \item \textbf{Initialization}: Start with a feasible solution \(\mathbf{x_0} \) and an initial block \(\mathcal{H}_0 \subseteq \{1, \dots, N\}\).
    \item \textbf{Subproblem Optimization}: For the current solution \(\mathbf{x^k}\), solve the subproblem restricted to the working set \(\mathcal{H}_k\) to find the descent direction \(\mathbf{d^k}  \in D(\mathbf{x^k}; \mathcal{H}_k)\). Update the solution as:
    \[
    \mathbf{x^{k+1}} =\mathbf{x^k}  +\mathbf{d^k}.
    \]
    \item \textbf{Optimality Check}: If optimality is achieved, terminate the algorithm. Else adjust the block \(\mathcal{H}_k\) for the next iteration using predefined rules or adaptive heuristics.
\end{itemize}

The existing BCD method~\cite{xie2023randomized}, randomly selects block coordinates based on certain properties of specific matrices.
In~\cite{xie2023randomized}, the authors propose the Accelerated Random Block Coordinate Descent algorithm and establish the following convergence theorem.

\begin{theorem}
The gap between the current objective function value and the optimal objective function value converges with rate \( 1 - v \), i.e.,

\[
\mathbf{c}^{\top} \mathbf{x}^{k+1} - \mathbf{c}^{\top} \mathbf{x}^* \leq (1 - v)(\mathbf{c}^{\top} \mathbf{x}^k - \mathbf{c}^{\top} \mathbf{x}^*).
\]
Here, \( v \) satisfies
\[
v \geq \zeta \frac{n ( \tilde{p} - 2 )}{(n^2 - 3)(n!)^2},
\]
where \( \zeta \) is a predefined probability parameter in the range \( (0,1) \), and \( \tilde{p} \) is a parameter representing the properties of the matrix in the range \( [3, n] \).
\end{theorem}

These approaches provide a simple mechanism for constructing subproblems by selecting block coordinates. However, when the size n of the OT problems is large, \( v \) becomes extremely small, leading to slow convergence.
Furthermore, these approaches do not account for the sparsity of the solution of the OT problem. Many variables that ultimately take zero values continue to be included in the block coordinates at each iteration. Moreover, these approaches typically provide only \emph{asymptotic} convergence guarantees,
rather than finite termination at an exact optimal solution.

\subsection{Sinkhorn Algorithm for Optimal Transport}

Instead of solving the classical OT problem defined in (1), the Sinkhorn algorithm~\cite{cuturi2013sinkhorn} introduces an entropy regularization term
\[
\min_{X \in \mathbb{R}_{+}^{n \times m}} \; \sum_{i=1}^{n} \sum_{j=1}^{m} c_{ij} x_{ij} - \varepsilon H(X),
\]
where \( \varepsilon > 0 \) is the regularization parameter, and \( H(X) \) is the entropy function defined as
\[
H(X) = -\sum_{i=1}^{n} \sum_{j=1}^{m} x_{ij} \log x_{ij}.
\]
The parameter \( \varepsilon \) controls the strength of the entropy regularization, balancing the trade-off between the classical transport cost and the entropy of the transport plan.

\noindent
Introduce Lagrange multipliers \( \alpha = (\alpha_i)_{i=1}^n \) and \( \beta = (\beta_j)_{j=1}^m \) for the supply and demand constraints, respectively. The Lagrangian function of the regularized OT problem is
\[
\mathcal{L}(X, \alpha, \beta) = \sum_{i=1}^{n} \sum_{j=1}^{m} c_{ij} x_{ij} - \varepsilon H(X) 
+ \sum_{i=1}^{n} \alpha_i \left( p_i - \sum_{j=1}^{m} x_{ij} \right)
+ \sum_{j=1}^{m} \beta_j \left( q_j - \sum_{i=1}^{n} x_{ij} \right).
\]
Taking the partial derivative of \( \mathcal{L} \) with respect to \( x_{ij} \) and setting it to zero, we obtain the optimality condition
\[
\frac{\partial \mathcal{L}}{\partial x_{ij}} = c_{ij} + \varepsilon (\log x_{ij} + 1) - \alpha_i - \beta_j = 0.
\]
Solving for \( x_{ij} \), we find
\[
x_{ij} = \exp\left( \frac{\alpha_i + \beta_j - c_{ij}}{\varepsilon} - 1 \right).
\]

\noindent
Let \( u_i = e^{\alpha_i / \varepsilon} \) and \( v_j = e^{\beta_j / \varepsilon} \), then the optimal transport plan can be expressed as
\[
x_{ij} = u_i e^{-c_{ij} / \varepsilon} v_j.
\]
Define \( K = e^{-C / \varepsilon} \) as the matrix obtained by element-wise exponentiating the cost matrix \( C = \{c_{ij}\} \). The transport plan \( X \) can then be written as:
\[
X = \text{diag}(\mathbf{u}) K \text{diag}(\mathbf{v}),
\]
where \( \text{diag}(\mathbf{u}) \) and \( \text{diag}(\mathbf{v}) \) are diagonal matrices whose diagonal entries are the elements of the vectors \( \mathbf{u} \) and \( \mathbf{v} \), respectively.

\noindent 
From the equality constraints \( X \mathbf{1}_m = \mathbf{p} \) and \( X^\top \mathbf{1}_n = \mathbf{q} \), we derive the following update rules
\[
\mathbf{u}^{(t+1)} = \frac{\mathbf{p}}{K \mathbf{v}^{(t)}},
\]
\[
\mathbf{v}^{(t+1)} = \frac{\mathbf{q}}{K^\top \mathbf{u}^{(t+1)}}.
\]
Here, the division is performed element-wise. 

\noindent
The Sinkhorn algorithm terminates when the change in the scaling factors \( \mathbf{u} \) and \( \mathbf{v} \) becomes sufficiently small
\[
\| \mathbf{u}^{(t)} - \mathbf{u}^{(t-1)} \| < \delta, \quad \| \mathbf{v}^{(t)} - \mathbf{v}^{(t-1)} \| < \delta,
\]
where \( \delta > 0 \) is a small predefined tolerance.

\noindent
The regularization parameter \( \varepsilon \) introduces a smoothing effect in the OT problem. As \( \varepsilon \) decreases, the solution becomes closer to the exact OT plan, but the algorithm may converge more slowly. Conversely, larger values of \( \varepsilon \) result in faster computations, but with a higher approximation error in the transport plan.

\section{Block Coordinate Network Simplex Methods with Basis Variable Succession}

The goal of the NS method is to identify the optimal basis variables through iterative pivoting. While each NS iteration efficiently updates the basis structure, computing reduced costs for all variables incurs a cost of \( O(N) \) per iteration.

To reduce this per-iteration cost, we adopt the BCD strategy: at each iteration, we restrict attention to a selected subset \(\mathcal{H}_k \subseteq \mathcal{N} \) of variables and solve a corresponding subproblem using the NS method. 

The OT problem is particularly well-suited to such block decomposition. In general LPs with \(M\) constraints, computing the reduced cost for a single variable requires \(O(M)\) operations. However, due to the special network structure of the OT problem, each reduced cost can be computed in \(O(1)\) time via the relation
\[
r_{ij} = c_{ij} - (\pi_i - \pi_{n + j}),
\]

\medskip
\noindent
We now summarize the notation and assumptions used throughout this section:
\begin{itemize}
    \item \( \mathbf{x}^k \in \mathbb{R}^N \): the basic feasible solution at iteration \(k\),
    \item \( \mathbf{x}^* \):the optimal solution to the OT problem~\eqref{eq:OT_vector},
    \item \( \mathcal{B}_k \subseteq \{1,\dots,N\} \): the set of basis indices at iteration \(k\),
    \item \( \mathcal{B}^\ast \subseteq \{1,\dots,N\} \): the basis of a globally optimal solution,
    \item \( \mathcal{H}_k \subseteq \{1,\dots,N\} \): the block of variable indices selected at iteration \(k\),
    \item \( \bar{c}^k := \bar{c}(\mathbf{x}^k, \mathcal{B}_k) \): the reduced cost vector at \( \mathbf{x}^k \).
\end{itemize}

\subsection{Basis Variables Succession}

In BCD schemes for the OT problem, the convergence behaviour hinges on the construction
of the working set \(\mathcal H_{k}\) at each iteration \(k\).

The following theorem clarifies when solving a subproblem with the NS method
restricted to \(\mathcal H_k\) can yield a globally optimal solution of the full OT problem,
and when this is impossible.

\begin{theorem}\label{thm:NS-global}
Apply the NS method to the subproblem restricted to \(\mathcal H_{k}\), and set
\[
d^{k}\in\mathcal D\bigl(\mathbf{x}^{k};\mathcal H_{k}\bigr),
\qquad
\mathbf{x}^{k+1}=\mathbf{x}^{k}+d^{k}.
\]
Then:
\begin{enumerate}
  \item[\textnormal{(i)}]
  If \(\mathcal B_k\subseteq \mathcal H_k\) and \(\mathcal B^\ast\subseteq \mathcal H_k\),
  then \(\mathbf{x}^{k+1}\) is a globally optimal solution of the full OT problem.

  \item[\textnormal{(ii)}]
  Assume that the global optimal solution \(\mathbf{x}^\ast\) is \emph{unique}.
  If \(\mathcal{B}_{k}\not\subseteq\mathcal H_{k}\),
  then even when \(\mathcal B^\ast\subseteq \mathcal H_k\),
  the iterate \(\mathbf{x}^{k+1}\) cannot be globally optimal.
\end{enumerate}
\end{theorem}

\begin{proof}
\medskip
\noindent\textbf{(i)}
Because \(\mathcal B_k\subseteq \mathcal H_k\), the current basic feasible solution
\(\mathbf{x}^k\) is feasible for the restricted subproblem, and thus the NS method can be
initialized from \(\mathbf{x}^k\) without losing feasibility.
Let \(\mathbf{x}^{k+1}\) be an optimal solution of the restricted subproblem.
Then \(\mathbf{x}^{k+1}\) is feasible for the full OT problem (it satisfies \(A\mathbf{x}^{k+1}=b\) and
\(\mathbf{x}^{k+1}\ge 0\)), hence
\[
c^\top \mathbf{x}^{k+1}\ \ge\ c^\top \mathbf{x}^\ast.
\]
On the other hand, since \(\mathcal B^\ast\subseteq \mathcal H_k\), the basic feasible solution
\(\mathbf{x}^\ast\) (with nonbasic variables equal to zero) is feasible for the restricted subproblem as well.
By optimality of \(\mathbf{x}^{k+1}\) for the restricted subproblem, we have
\[
c^\top \mathbf{x}^{k+1}\ \le\ c^\top \mathbf{x}^\ast.
\]
Combining the two inequalities yields \(c^\top \mathbf{x}^{k+1}=c^\top \mathbf{x}^\ast\), so
\(\mathbf{x}^{k+1}\) attains the global optimal value of the full OT problem. Therefore
\(\mathbf{x}^{k+1}\) is globally optimal.

\medskip
\noindent\textbf{(ii)}
Assume \(\mathbf{x}^\ast\) is the unique globally optimal solution.
Let \(j\notin \mathcal H_k\) with \(x^k_j>0\). By the definition of the restricted subproblem, the \(j\)-th component cannot change:
\[
x^{k+1}_j \;=\; x^k_j + d^k_j \;=\; x^k_j \;>\; 0.
\]
Now suppose (for contradiction) that \(\mathbf{x}^{k+1}\) is globally optimal.
By uniqueness of the global optimizer, this would imply \(\mathbf{x}^{k+1}=\mathbf{x}^\ast\),
and in particular \(x^\ast_j=x^{k+1}_j>0\).

However, under the assumption \(\mathcal B^\ast\subseteq \mathcal H_k\), we have \(j\notin \mathcal B^\ast\),
and since \(\mathcal B^\ast\) is a basis of a basic feasible optimal solution, all variables outside
\(\mathcal B^\ast\) are nonbasic and therefore equal to zero. Thus \(x^\ast_j=0\), contradicting
\(x^\ast_j>0\). Consequently, \(\mathbf{x}^{k+1}\neq \mathbf{x}^\ast\), and \(\mathbf{x}^{k+1}\) cannot be globally optimal.
\end{proof}

The above result motivates a deterministic working set construction rule, which we call
\emph{basis-variable succession}. Specifically, at each iteration $k$, we choose the block
$\mathcal{H}_k$ so that it contains the current basis, i.e., $\mathcal{B}_k \subseteq \mathcal{H}_k$.
This choice contrasts with existing BCD-type approaches~\cite{xie2023randomized}, where the
working set is generated without explicitly preserving the current basis.

From a computational perspective, let $L_k$ denote the number of NS pivots
required to solve the restricted subproblem over $\mathcal{H}_k$, and let $L$ denote the pivot
count when applying the NS method to the full problem. Then the dominant cost of the
restricted solve scales as $O(L_k\,|\mathcal{H}_k|)$, which can be substantially smaller than the
$O(LN)$ cost incurred by scanning all $N$ variables in each pivot of the full problem.

This strategy has two immediate practical benefits. First, since the
restricted subproblem retains all equality constraints and contains at least $M-1$ variables,
it admits a feasible spanning tree basis and hence remains well-posed throughout the
iterations. Second, reusing $\mathcal{B}_k$ as a warm start provides a readily available initial
basis for the next subproblem, avoiding the overhead of repeatedly constructing an initial
feasible basis from scratch.

\subsection{A Prototype of Block Coordinate Descent Network Simplex Methods and its Finite Termination}

Following the strategy of section 4.1, choose a block
\[
\mathcal H_{k}\supseteq \mathcal{B}_{k},\qquad 
\lvert\mathcal H_{k}\rvert\ge M-1,
\]
and split it into
\[
\mathcal{B}_{k},\qquad
\mathcal I_{k}=\mathcal H_{k}\setminus\mathcal{B}_{k}.
\]

\vspace{2pt}\noindent
If \(\mathcal I_{k}\) is picked arbitrarily, the subproblem
\[
\mathcal D\!\bigl(\mathbf{x}^{k};\mathcal H_{k}\bigr)
\]
might admit only the zero direction, yielding the same basis
\(\mathcal{B}_{k+1}=\mathcal{B}_{k}\) and \(\mathbf{x}^{k}=\mathbf{x}^{k+1}\).

\medskip\noindent
Define the index set of negative reduced costs
\[
\mathcal N_{k}
  \;=\;
  \bigl\{\,\ell\in\{1,\dots,nm\}:\;
           \bar c^{\,k}_{\ell}<0\bigr\}.
\]
From iteration \(k\) on we impose the rule
\begin{equation}\label{eq:descent-block}
  \mathcal I_{k}\cap\mathcal N_{k}\;\neq\;\varnothing.
\end{equation}
If \(\mathcal N_{k}=\varnothing\), then
\(\bar c^{\,k}\ge0\) and, by complementary slackness,
\(\mathbf{x}^{k}\) is already globally optimal.

\medskip
Let $q(\mathbf{x}^k;\mathcal{H}_k)$ denote the optimal value of the restricted subproblem
$\mathcal{D}(\mathbf{x}^k;\mathcal{H}_k)$. The next theorem establishes finite termination and global optimality of the resulting sequence \(\{\mathbf{x}^{k}\}\).

\begin{theorem}\label{thm:proto_monotone_finite}
Consider the prototype BCDNS method with the block selection rule~\eqref{eq:descent-block}.
The sequence $\{q(\mathbf{x}^k;\mathcal{H}_k)\}_{k\ge 0}$ is non-increasing monotone.
Moreover, if the NS subproblem solver is equipped with Bland's anti-cycling pivot rule
~\cite{bland1977new}, then the method after finitely many degenerate pivots, either a strictly improving pivot occurs or the method terminates. In particular, the algorithm terminates in finitely many iterations and returns a globally optimal OT solution.
\end{theorem}

\begin{proof}
Fix an outer iteration $k$ and consider the working set $\mathcal{H}_k=\mathcal{B}_k\cup\mathcal{I}_k$.
We solve the restricted subproblem $\mathcal{D}(\mathbf{x}^k;\mathcal{H}_k)$ by the NS method started on the current basis $\mathcal{B}_k$.

\medskip
\noindent\textbf{Monotonicity (non-increase).}
A basic property of the NS method is that each pivot maintains feasibility and does not
increase the objective value; hence the optimal value of the restricted subproblem satisfies
\begin{equation}\label{eq:mono_noninc_thm}
q(\mathbf{x}^k;\mathcal{H}_k)\;\le\; q(\mathbf{x}^k;\mathcal{B}_k).
\end{equation}
Since $\mathcal{B}_k$ is the optimal basis obtained from the previous restricted subproblem, the start objective equals the previous subproblem optimum:
\begin{equation}\label{eq:warmstart_value_thm}
q(\mathbf{x}^k;\mathcal{B}_k)\;=\; q(\mathbf{x}^{k-1};\mathcal{H}_{k-1}).
\end{equation}
Combining~\eqref{eq:mono_noninc_thm}--\eqref{eq:warmstart_value_thm} yields
\[
q(\mathbf{x}^k;\mathcal{H}_k)\;\le\; q(\mathbf{x}^{k-1};\mathcal{H}_{k-1}),
\]
so $\{q(\mathbf{x}^k;\mathcal{H}_k)\}$ is non-increasing monotone.

\medskip
\noindent\textbf{Finite termination and eventual strict decrease under Bland's rule.}
Degenerate pivots may yield zero objective improvement, so strict decrease cannot be guaranteed at every pivot.
However, when Bland's anti-cycling pivot rule is used in the NS solver~\cite{bland1977new},
cycling is prevented and hence a basis cannot repeat. Since the number of feasible bases is finite,
each restricted NS solve terminates in finitely many pivots; moreover, it cannot perform
infinitely many zero-improvement pivots, i.e., after finitely many degenerate pivots it either makes a
strictly improving pivot or reaches restricted optimality.

At the outer iteration, if the negative reduced cost set $\mathcal{N}_k$ is empty, then all reduced costs are
nonnegative and $\mathbf{x}^k$ is globally optimal, so the method terminates.
Otherwise, $\mathcal{N}_k\neq\varnothing$ and the block selection rule~\eqref{eq:descent-block} guarantees that the
restricted subproblem contains at least one negative reduced cost index; therefore the NS solver of the restricted subproblem at iteration $k$ is well defined and
terminates after finitely many pivots.
Consequently, the overall method terminates in finitely many outer iterations.

\end{proof}


\subsection{Random Selection Block Coordinate Descent Network Simplex Method}
\label{subsec:rs-bcdns}

A natural way to instantiate the prototype in section 4.2 is to
construct the working set by random sampling.
we construct at iteration $k$ a working set
$\mathcal{H}_k \supseteq \mathcal{B}_k$ by augmenting the current basis
$\mathcal{B}_k$ with a randomly sampled subset $\mathcal{I}_k$.
The key requirement inherited from~\eqref{eq:descent-block} is that the block
must contain at least one index with negative reduced cost whenever the current
iterate is not yet globally optimal.

\medskip

Fix a sampling ratio $s \in (0,1)$ and set $b:=\lvert \mathcal{B}_k\rvert = M-1$.
At iteration $k$, we draw $\lvert \mathcal{I}_k\rvert = \lceil sN\rceil$ indices
uniformly at random from the complement of the current basis,
$\{1,\dots,N\}\setminus \mathcal{B}_k$.
We compute the reduced costs of the sampled indices and enforce the descent condition
\[
\mathcal{I}_k \cap \mathcal{N}_k \neq \varnothing,
\qquad
\mathcal{N}_k:=\{\ell: \bar c^k_\ell<0\}.
\]
Then we form
\[
\mathcal{H}_k := \mathcal{B}_k \cup \mathcal{I}_k
\quad\text{and solve}\quad
\mathcal{D}(\mathbf{x}^k;\mathcal{H}_k)
\ \text{by the NS method.}
\]

\medskip

The algorithm is presented below.

\begin{algorithm}[H]
\caption{Random Selection Block Coordinate Descent Network Simplex Method (RS-BCDNS)}
\label{alg:rs-bcdns}
\begin{algorithmic}
\State \textbf{Initialization:}
Find an initial feasible solution $\mathbf{x}^0$ and an initial basis $\mathcal{B}_0$
Choose a sampling ratio $s\in(0,1)$.
\For{$k=0,1,2,\dots$}
    \State Compute reduced costs for all indices. If {$\bar c^k_\ell \ge 0$ for all $\ell=1,\dots,N$} \textbf{terminate} 
    \Repeat
        \State Sample $\mathcal{I}_k \subseteq \{1,\dots,N\}\setminus\mathcal{B}_k$
        uniformly at random with $\lvert \mathcal{I}_k\rvert=\lceil sN\rceil$.
        \State Compute $\bar c^k_\ell$ for all $\ell\in\mathcal{I}_k$.
    \Until{$\min_{\ell\in\mathcal{I}_k}\bar c^k_\ell < 0$}
    \State Set $\mathcal{H}_k \gets \mathcal{B}_k \cup \mathcal{I}_k$.
    \State 
    Solve $\mathcal{D}(\mathbf{x}^k;\mathcal{H}_k)$ by the NS method from $\mathcal{B}_k$, obtaining $d^k$ and the updated basis $\mathcal{B}_{k+1}$.
    \State Update $\mathbf{x}^{k+1}\gets \mathbf{x}^k + d^k$.
\EndFor
\end{algorithmic}
\end{algorithm}

\subsection{Grouped Selection Block Coordinate Descent Network Simplex method}

In section 4.3, we introduced a random selection variant as a natural baseline. Although this random augmentation is simple and preserves the basis-succession mechanism, it can be highly inefficient for large-scale OT instances. Indeed, optimal transport plans are typically sparse, and only a small fraction of variables eventually belong to an optimal basis. Consequently, sampling $\mathcal{I}_k$ uniformly at random tends to include many variables that are unlikely to become basis, so their reduced costs are evaluated repeatedly and they may participate in subproblems without contributing to meaningful progress.

To this end, we propose a grouped block-coordinate selection strategy that partitions the index set $\{1,\dots,N\}$ into two disjoint subsets:
\begin{itemize}
  \item $\mathcal{R}$: indices that are deemed likely to enter an optimal basis $\mathcal{B}^\ast$,
  \item $\mathcal{S}$: indices with low likelihood of being part of $\mathcal{B}^\ast$.
\end{itemize}

Let \(0 < s < t \le 1\) be fixed sampling parameters. At each iteration, the reduced costs are computed for a randomly selected subset of size \(tN\) from \(\mathcal{R}\), together with a small subset from \(\mathcal{S}\) for exploration. Among these, we select the \(sN\) variables with the smallest reduced costs to form the block \(\mathcal{I}_k\), as these variables are most promising for descent directions.

If all selected variables have non-negative reduced costs, the algorithm infers that none are useful for further improvement and moves them to set \(\mathcal{S}\). The selection process is then repeated until condition~\eqref{eq:descent-block} is satisfied.

This strategy guides the search toward structurally relevant directions and improves computational efficiency. Based on this mechanism, we introduce the \textbf{Grouped Selection Block Coordinate Descent Network Simplex (GS-BCDNS)} method. The full algorithm is presented below.

\begin{algorithm}[H]
\caption{Grouped Selection Block Coordinate Descent Network Simplex Method (GS-BCDNS)}
\label{alg:bcd_nsm}
\begin{algorithmic}
\State \textbf{Initialization:} 
\State Find an initial feasible solution \( X^0 \in \mathbb{R}^{n \times m} \). Set the initial basis variables as \( \mathcal{J}_0 \). Choose parameters \( s, t \in (0, 1) \) such that \( s < t \).
Initialize \( \mathcal{S} = \emptyset \) and \( \mathcal{R} = \{1, 2, \ldots, N\} \). Shuffle \( \mathcal{R} \) randomly. Define the threshold sequence \( \{ e_k \} \). 

\For{$k = 0, 1, 2, \ldots$}
    \State \textbf{Step 1:} Select the first \( \min(|\mathcal{R}|, tN) \) variables from \( \mathcal{R} \) and \( \min(|\mathcal{S}|, 0.1tN) \) variables from \( \mathcal{S} \). Compute the reduced costs for all selected variables and sort them in ascending order. Select the \( sN \) variables with smallest reduced costs to form \( \mathcal{I}_k \). If all selected variables have non-negative reduced costs, move them to \( \mathcal{S} \) and repeat Step 1.
    
    \State \textbf{Step 2:} Move the non-selected variables to \( \mathcal{S} \).
    
    \State \textbf{Step 3:} Set \( \mathcal{H}_k = \mathcal{I}_k \cup \mathcal{B}_k \).
    
    \State \textbf{Step 4:} Use the NS method to solve the subproblem \(\mathcal{D}(x^k; \mathcal{H}_k) \) and obtain \( d^k \in \mathcal{D}(x^k; \mathcal{H}_k) \). Update the basis to \( \mathcal{B}_{k+1} \). Move variables with reduced cost \( r_{ij} < e_k \) to \( \mathcal{R} \), and move the remaining variables to \( \mathcal{S} \).
    
    \State \textbf{Step 5:} Update \( x^{k+1} := x^k + d^k \).
    
    \State \textbf{Step 6:} \textbf{Optimality check:} If $\mathcal{R}=\emptyset$, compute reduced costs for all variables.
    \If{$r_{ij}\ge 0$ for all $i,j$}
        \State \textbf{terminate}.
    \Else
        \State Update $\mathcal{R}\gets \{(i,j): r_{ij}<0\}$.
        \State \textbf{return to Step 1}.
    \EndIf
    
\EndFor    
\end{algorithmic}
\end{algorithm}

Let \( \tilde{k} \) denote the total number of iterations, and let \( l_k \) represent the number of NS iterations required to solve the subproblem at iteration \( k \). The time complexity of the GS-BCDNS method is given by:
\[
O\left( \tilde{k} tN + \sum_{k=1}^{\tilde{k}} l_k sN \right).
\]
The complexity clearly depends on the choice of sampling parameters \( s \) and \( t \). Decreasing their values reduces the cost per iteration, but increases the total number of iterations \( \tilde{k} \). Thus, careful tuning of \( s \) and \( t \) is essential to balance convergence speed and efficiency.

In each iteration, the space complexity of the GS-BCDNS method is:
\[
O(|\mathcal{H}_k|) = O(sN + M - 1).
\]
Compared to the Sinkhorn algorithm, which requires \( O(N) \) space, the GS-BCDNS method is more memory-efficient.

\section{Numerical Experiments}\label{sec:experiments}

In this section, we report numerical experiments designed to (i) study the effect of the block-size parameters $(s,t)$ in the proposed BCDNS framework, (ii) compare the proposed methods with representative baselines on the Wasserstein distance computations, and (iii) evaluate the performance and scalability of the proposed methods on large-scale OT instances.\footnotetext[1]{Numerical experiments were conducted on a computer with an Intel Core i7-12700H CPU (2.70\,GHz) and 16\,GB of RAM. The implementation was performed in Python 3.8. The code and the scripts for reproducing all figures are publicly available at \url{https://github.com/842422492/BCDNS.git}.}

\subsection{Experimental Setup}\label{subsec:exp-setup}

This section describes the common procedure used to generate OT test instances and the
implementation details shared by all methods in Section 5.2--Section 5.4.
Unless otherwise stated, we consider balanced OT with equal numbers of supply and demand nodes, i.e., $n=m$.

\paragraph{Test instance generation and reproducibility.}
All Wasserstein test instances are constructed by Monte Carlo discretization of continuous
probability measures.
For a prescribed sample size $n$, we draw i.i.d.\ samples
$\{u_i\}_{i=1}^n$ from a source distribution $\mu$ and $\{v_j\}_{j=1}^n$ from a target
distribution $\nu$.
We define the quadratic transport cost
\begin{equation}\label{eq:cost-matrix}
  c_{ij} \;=\; \|u_i - v_j\|^2, \qquad 1\le i,j\le n,
\end{equation}
and set the discrete marginals to be uniform,
\begin{equation}\label{eq:uniform-marginals}
  p=q=\Bigl(\tfrac{1}{n},\dots,\tfrac{1}{n}\Bigr)^\top \in \mathbb{R}^n.
\end{equation}
The resulting linear program is then solved by the competing methods.
To ensure reproducibility, we save and version-control all sampled point sets
$\{u_i\}$ and $\{v_j\}$ for every reported instance; hence the corresponding cost matrix $C$
in~\eqref{eq:cost-matrix} can be reconstructed exactly.

\paragraph{Algorithmic settings.}
\begin{itemize}

  \item \textbf{Sinkhorn (stabilized).}
  Following~\cite{cuturi2013sinkhorn} and the stabilized implementation used in~\cite{xie2023randomized},
  we solve the entropically regularized OT problem with regularization coefficient
  $\gamma = \varepsilon/(4\log n)$ as suggested in~\cite{dvurechensky2018computational}.
  We test $\varepsilon \in \{10^{-4},10^{-3},10^{-2},10^{-1}\}$.
  Each Sinkhorn iterate is rounded to feasibility using the projection/rounding procedure
  (Algorithm~2 in~\cite{altschuler2017near}).
  To improve numerical stability, all updates are carried out in log-domain via the LogSumExp
  operator.
  The algorithm terminates when the absolute change in objective value falls below
  $\delta = 10^{-6}$ or when the runtime exceeds $2000$ seconds.

  \item \textbf{GS-BCDNS.}
  
    Our NS method is implemented by following the design of the
    network simplex solver in the Python \texttt{NetworkX} package.
    In the experiments, we multiply the cost matrix by a positive scaling factor and round to integers
    to obtain an equivalent integer-cost instance for the NS pivots; this scaling does not change the
    set of optimal transport plans, since it preserves the ordering of feasible solutions by objective value.
    In Section 5.2, we first perform a search on $(s,t)$ to identify an efficient regime for subproblem sizes.
    Based on the empirical findings in Section 5.2, we fix $(s,t)$ for all subsequent comparisons.
    Unless otherwise stated, we use
    \[
      s=\tfrac{2}{n}, \qquad t=\tfrac{20}{n}.
    \]

  \item \textbf{RS-BCDNS.}
  The method is initialized by the Northwest Corner method and set the block size to 256.

\end{itemize}

\subsection{Effect of Parameters $s$ and $t$}\label{sec:st-effect}

In GS-BCDNS, the sampling parameters $s,t\in(0,1)$ determine the working-set size
$\lvert\mathcal H_k\rvert \approx (M-1) + sN$ and the number of candidate variables
whose reduced costs are inspected before forming $\mathcal I_k$.
This subsection investigates how $(s,t)$ affects computational cost in practice.

\paragraph{Test Problem 1. (Uniform--Normal, varying $s$ and $t$)}
We compute the Wasserstein distance between the uniform distribution on $[-1,1]$
and the standard normal distribution $\mathcal N(0,1)$ by drawing
$x_i \sim \mathrm{Unif}[-1,1]$ and $y_j \sim \mathcal N(0,1)$ using
$i.i.d.$ samples of equal size $n\in\{250,500,1000\}$.

For each pair $(s,t)$, we run GS-BCDNS until
the optimality test is satisfied and record:
(i) the total number of reduced-cost evaluations (denoted ``reduced-cost evals''),
(ii) the total number of pivots performed by the network simplex solver across all subproblems
(denoted ``pivot count''), and
(iii) wall-clock runtime in seconds.
The reduced-cost evaluations measure the dominant arithmetic workload, because each
evaluation is an $O(1)$ operation in the OT network structure, and they are performed
repeatedly during candidate screening and optimality checks.

\paragraph{Comment on Test Problem 1.}
Figures~\ref{fig:st-grid-250}--\ref{fig:st-grid-1000} visualize the three metrics as heatmaps over $(s,t)$. The white regions correspond to parameter pairs that were not evaluated (in particular, those that violated $s<t$).

Across all $n$, the runtime heatmap closely matches the reduced-cost evaluation heatmap,
indicating a strong positive correlation between these two quantities: configurations that
trigger more reduced-cost computations consistently take longer in wall-clock time.
Moreover, for all three problem sizes, both reduced-cost evaluations and runtime attain their smallest values when $s$ is chosen on the order of $2/n$ -- -$3/n$ and
$t$ is on the order of $20/n$ -- -$30/n$, i.e.,
\[
s \in \Bigl[\tfrac{2}{n},\,\tfrac{3}{n}\Bigr], \qquad
t \in \Bigl[\tfrac{20}{n},\,\tfrac{30}{n}\Bigr].
\]
suggesting that a \emph{small} working-set fraction together with a \emph{moderate} candidate
screening fraction offers the best trade-off between (a) selecting sufficiently good entering
candidates and (b) avoiding excessive reduced-cost screening overhead.

In contrast, the pivot count depends much more strongly on $s$ than on $t$.
As $s$ increases, the subproblem size grows and the network simplex solver needs fewer pivots
to achieve the optimality of the subproblem, leading to a clear decrease in pivot count.
However, a larger $s$ also increases the per-iteration burden through a larger $\lvert\mathcal H_k\rvert$
and more bookkeeping within each subproblem; consequently, beyond a certain threshold, the
reduced-cost workload increases again even though the pivot count continues
to fall. This shows that the pivot count alone is not a reliable proxy for the total runtime.

In general, qualitative patterns are consistent for $n=250,500,1000$, while absolute magnitudes increase with problem size as expected.

\begin{figure*}[t]
  \centering
  \includegraphics[width=\textwidth]{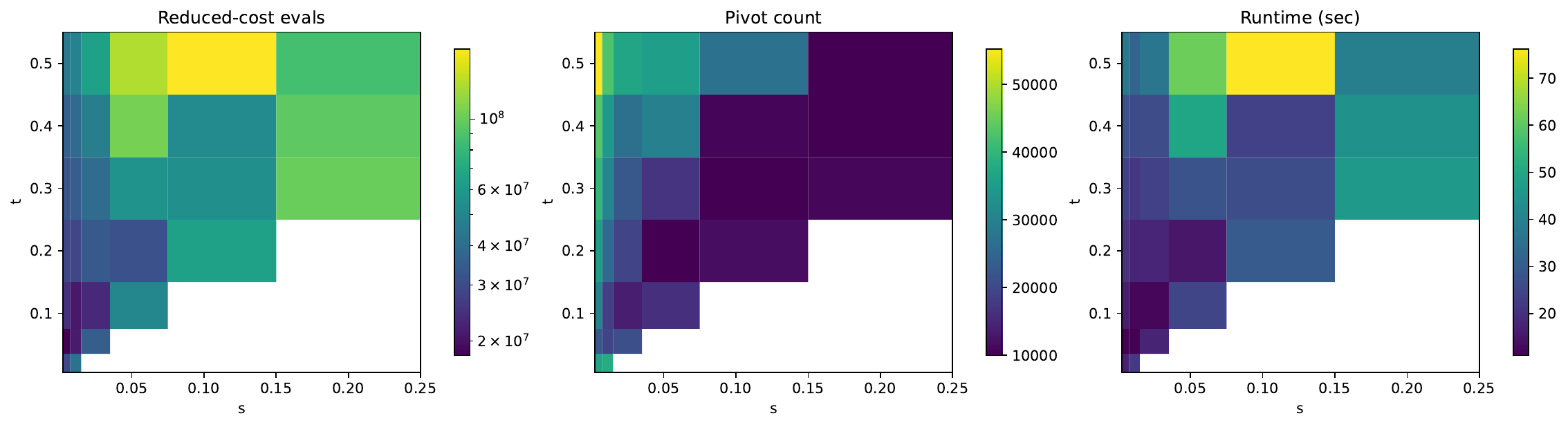}
  \caption{Effect of $(s,t)$ on reduced-cost evaluations, pivot count, and runtime for the instance with $n=250$.}
  \label{fig:st-grid-250}
\end{figure*}

\begin{figure*}[t]
  \centering
  \includegraphics[width=\textwidth]{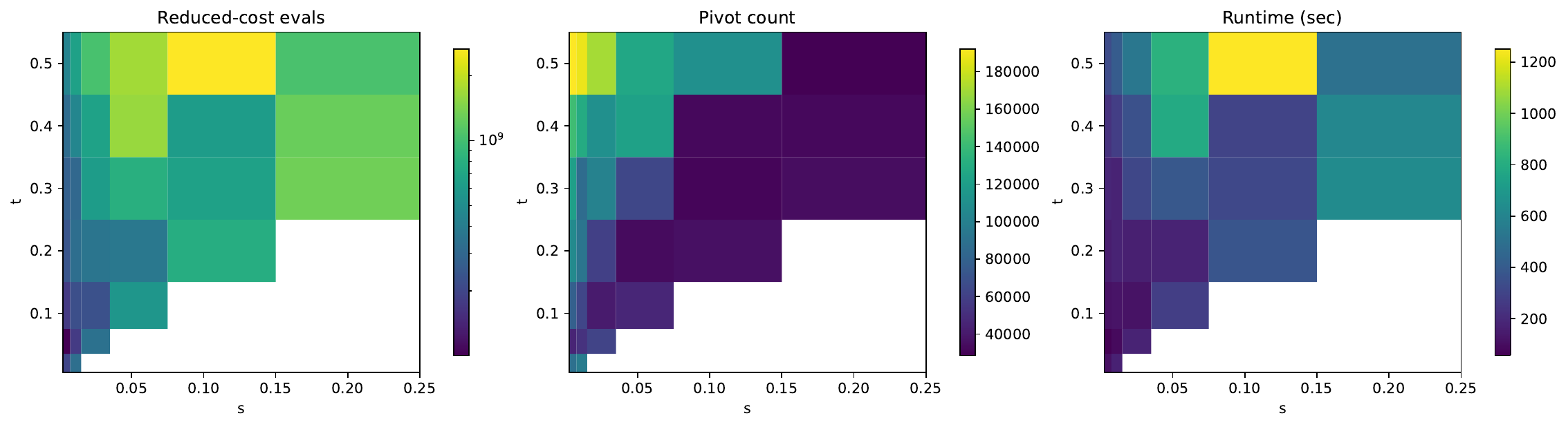}
  \caption{Effect of $(s,t)$ on reduced-cost evaluations, pivot count, and runtime for the instance with $n=500$.}
  \label{fig:st-grid-500}
\end{figure*}

\begin{figure*}[t]
  \centering
  \includegraphics[width=\textwidth]{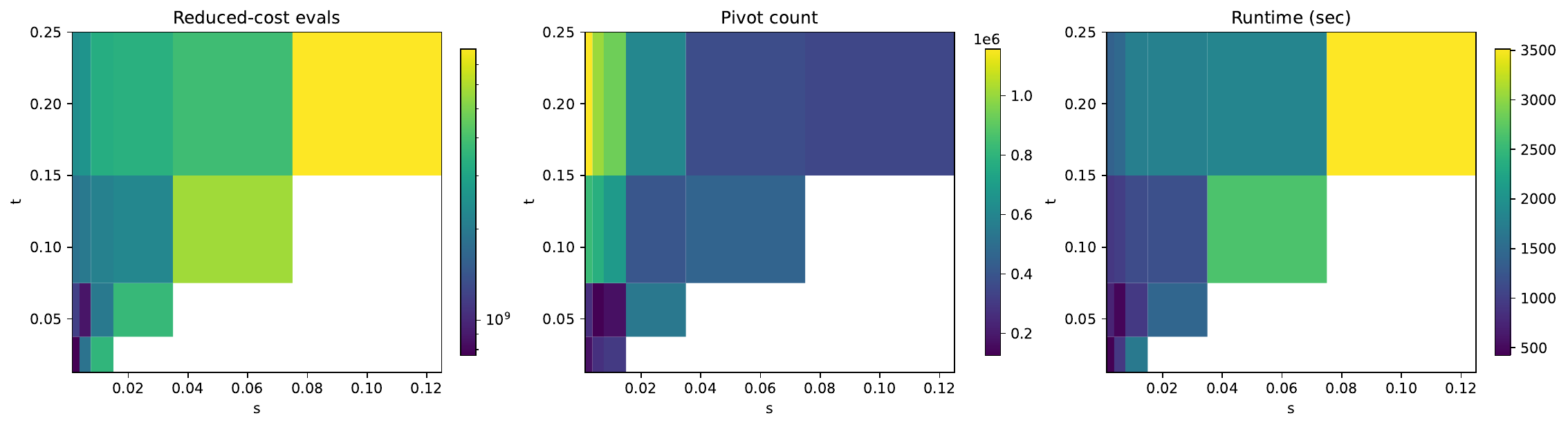}
  \caption{Effect of $(s,t)$ on reduced-cost evaluations, pivot count, and runtime for the instance with $n=1000$.}
  \label{fig:st-grid-1000}
\end{figure*}

\subsection{Comparison with the Existing Methods}\label{sec:exp-compare}

In this subsection, we compare the proposed \textbf{GS-BCDNS} with the existing method.

First, we compare the proposed \textbf{GS-BCDNS} with the classical
\textbf{NS} method and the baseline \textbf{RS-BCDNS} method on Test Problem 2 and 3.
We also report (i) the wall-clock runtime and (ii) the total number of reduced-cost evaluations. Consistent with this, the experiments show a strong correlation between the reduced-cost evaluations and the runtime. 
Second, We compare \textbf{GS-BCDNS} with the stabilized \textbf{Sinkhorn} algorithm on Test Problem 4.
For Sinkhorn, we test multiple regularization levels $\varepsilon\in\{10^{-1},10^{-2},10^{-3},10^{-4}\}$
under the same stabilization and stopping rule as in the previous experiments.
For evaluation, we report the \emph{optimality gap} $\mathrm{Gap}(t)=f(x(t))-f^*$ as a function of wall-clock time.

\paragraph{Test Problem 2 (Uniform--Normal, varying $n$).}
We compute the Wasserstein distance between the uniform distribution on $[-1,1]$
and the standard normal distribution $\mathcal N(0,1)$ by drawing
$x_i \sim \mathrm{Unif}[-1,1]$ and $y_j \sim \mathcal N(0,1)$.
We vary the problem size as $n\in\{50,100,150,200,250,300,350,400\}$. 

\paragraph{Test Problem 3 (Normal--Mixture, varying $n$).}
We compute the Wasserstein distance between $\mathcal N(0,1)$ and a two-component
Gaussian mixture
\[
\frac{1}{2}\,\mathcal N(-2,1) \;+\; \frac{1}{2}\,\mathcal N(2,1),
\]
by sampling $x_i\sim \mathcal N(0,1)$ and
$y_j \sim \frac{1}{2}\mathcal N(-2,1)+\frac{1}{2}\mathcal N(2,1)$.
The same range of $n$ is tested.

\paragraph{Comment on Test Problem 2.}
Figure~\ref{fig:bench-unifnorm} shows that \textbf{GS-BCDNS} consistently achieves the
smallest runtime and the fewest reduced-cost evaluations over all tested $n$.
Moreover, the advantage becomes more pronounced as $n$ increases:
GS-BCDNS reduces the number of reduced-cost evaluations by roughly an order of magnitude compared with NS.
In contrast, \textbf{RS-BCDNS} provides only modest improvements over NS.
The corresponding speedup curves in Figure~\ref{fig:speedup-unifnorm} indicate that
GS-BCDNS attains large and increasing speedups with respect to NS as $n$ grows,
whereas RS-BCDNS stays close to the baseline.

\paragraph{Comment on Test Problem 3.}
The Normal--Mixture instances are empirically more challenging for random block construction.
As shown in Figure~\ref{fig:bench-normmix}, \textbf{GS-BCDNS} remains the best-performing
method across all tested sizes, again substantially reducing reduced-cost evaluations and runtime.
However, \textbf{RS-BCDNS} can become \emph{slower} than NS for medium-to-large $n$.
This phenomenon is consistent with the mechanism of RS-BCDNS: random augmentation tends to
include many variables that are unlikely to belong to the final optimal basis, leading to repeated
subproblem solves with limited progress while still incurring significant reduced-cost computations.
The speedup plot in Figure~\ref{fig:speedup-normmix} confirms that RS-BCDNS may fall below
the NS baseline, whereas GS-BCDNS maintains a clear advantage.

\begin{figure}[t]
  \centering
  \includegraphics[width=\linewidth]{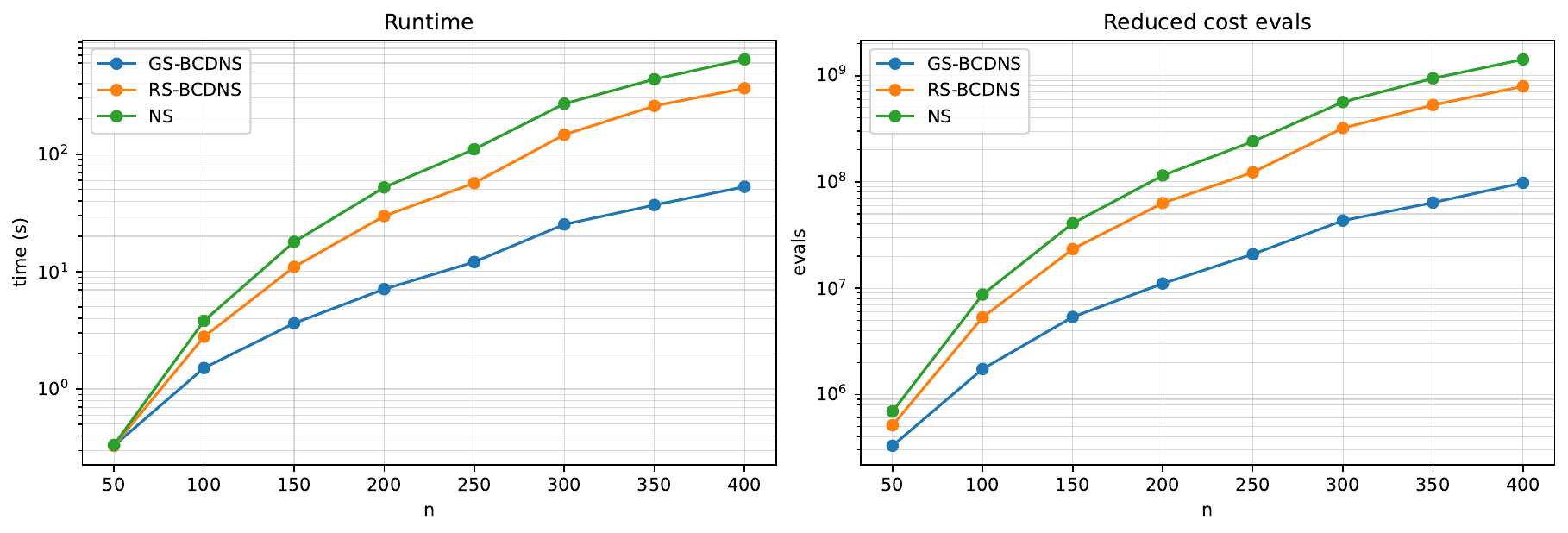}
  \caption{Test Problem 2: runtime (left) and total reduced-cost evaluations (right) for NS, RS-BCDNS, and GS-BCDNS.}
  \label{fig:bench-unifnorm}
\end{figure}

\begin{figure}[t]
  \centering
  \includegraphics[width=0.78\linewidth]{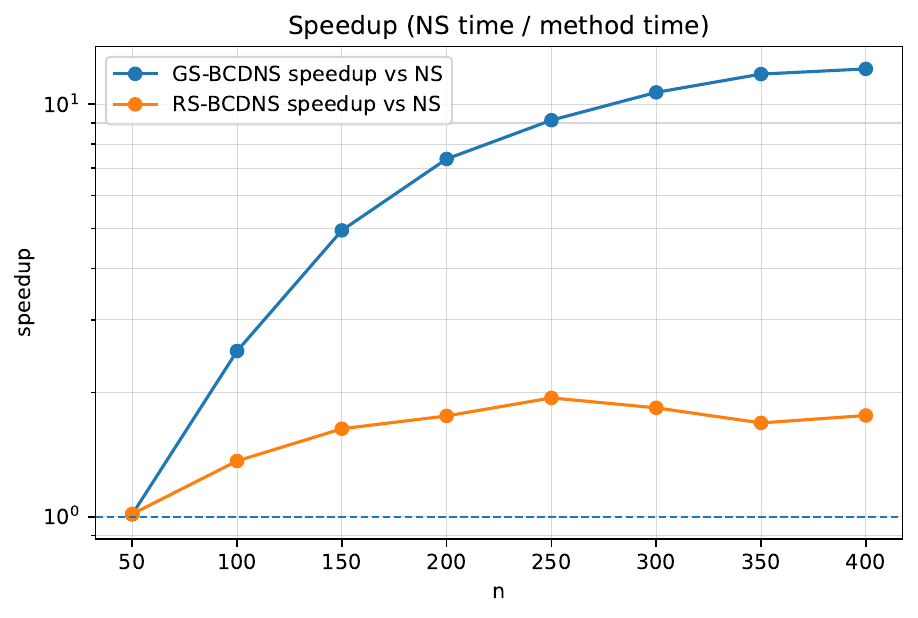}
  \caption{Test Problem 2: speedup relative to NS (NS time / method time).}
  \label{fig:speedup-unifnorm}
\end{figure}

\begin{figure}[t]
  \centering
  \includegraphics[width=\linewidth]{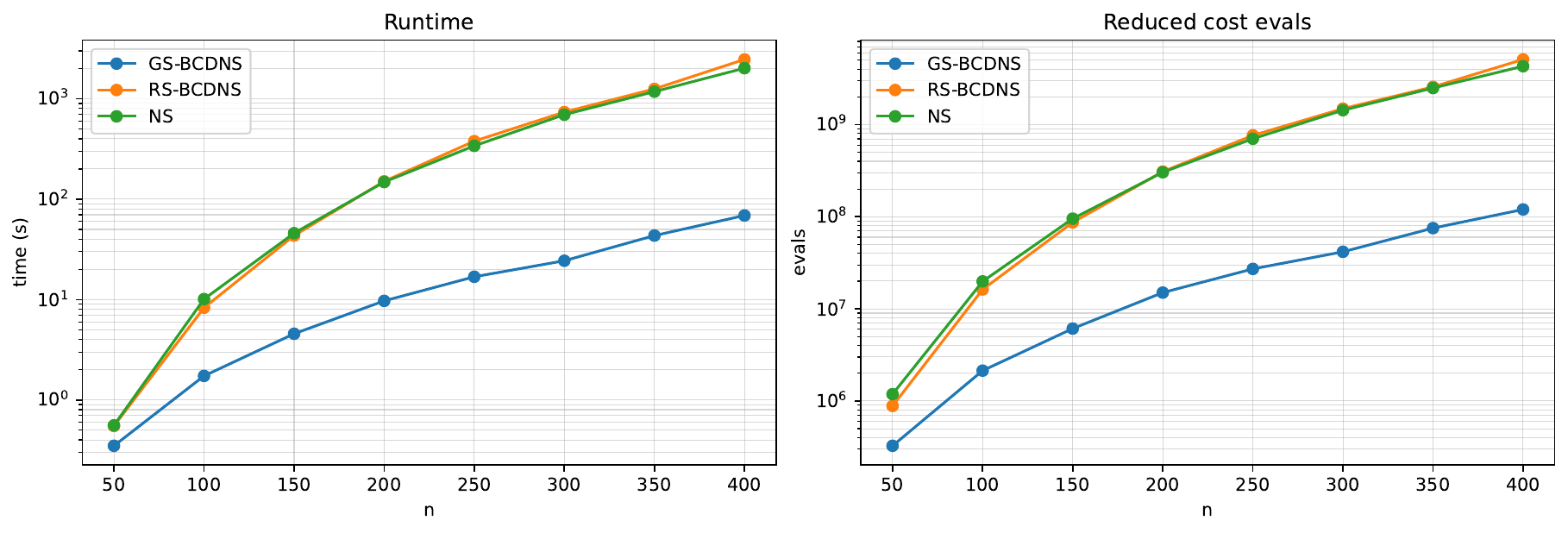}
  \caption{Test Problem 3: runtime (left) and total reduced-cost evaluations (right) for NS, RS-BCDNS, and GS-BCDNS.}
  \label{fig:bench-normmix}
\end{figure}

\begin{figure}[t]
  \centering
  \includegraphics[width=0.78\linewidth]{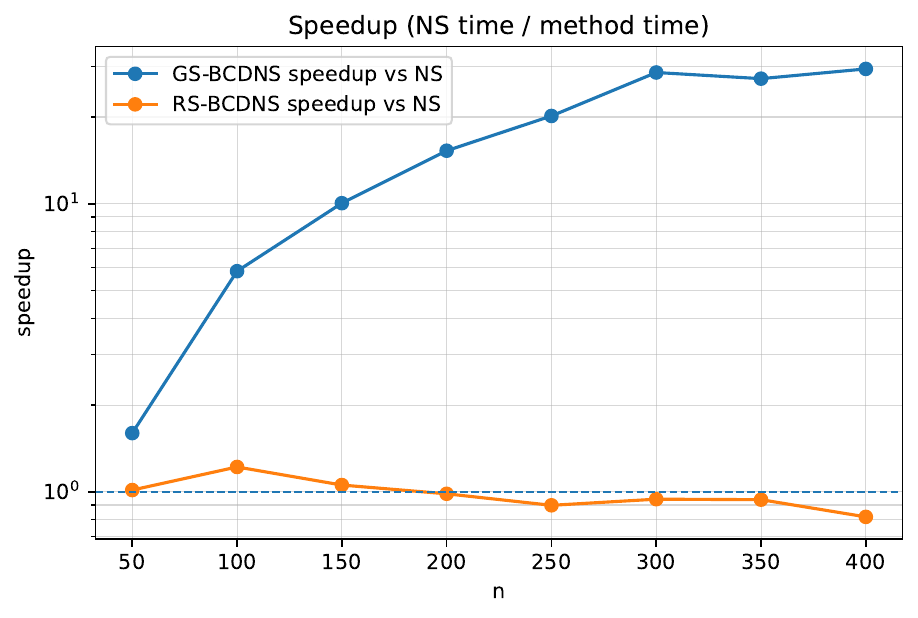}
  \caption{Test Problem 3: speedup relative to NS (NS time / method time).}
  \label{fig:speedup-normmix}
\end{figure}


\paragraph{Test Problem 4 (1D/2D Uniform--Normal).}
Wasserstein distance estimation between (i) a one-dimensional uniform distribution on $[-1,1]$ and a standard
normal distribution, and (ii) their two-dimensional counterparts on $[-1,1]^2$ and
$\mathcal N(0,I_2)$.

\paragraph{Comment on Test Problem 4.}
Figures~\ref{fig:tp4_gap_1d}--\ref{fig:tp4_gap_2d} show a consistent accuracy--speed trade-off.
With relatively large regularization, Sinkhorn quickly reduces the
objective at the beginning, but the gap stagnates at a non-negligible level due to the inherent regularization bias.
As $\varepsilon$ decreases, Sinkhorn becomes progressively closer to the unregularized OT solution, yet its convergence speed deteriorates substantially.
In contrast, GS-BCDNS exhibits a slower initial decrease than low-precision Sinkhorn, but it continues to make steady progress and ultimately attains the \emph{exact} optimal solution of the original OT problem.

\begin{figure}[t]
  \centering
  \includegraphics[width=0.8\linewidth]{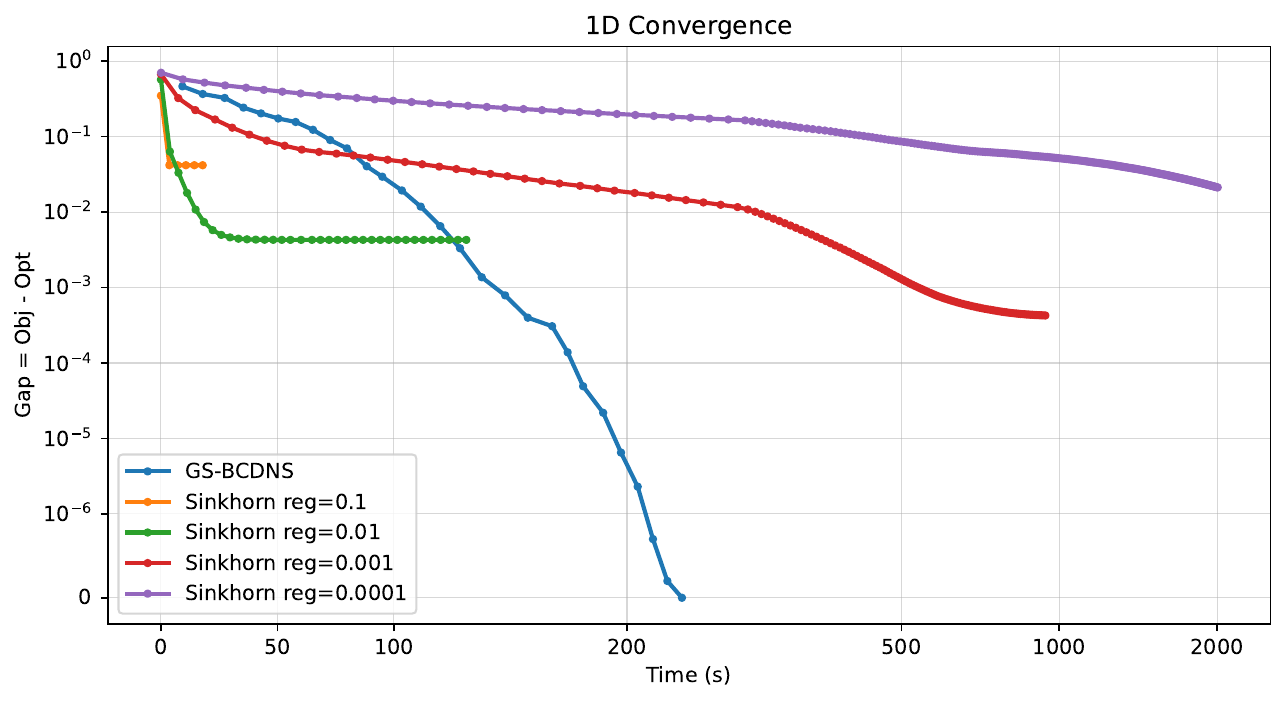}
  \caption{Test Problem 4 (1D): optimality gap versus time for GS-BCDNS and stabilized Sinkhorn under different
  regularization levels $\varepsilon$.}
  \label{fig:tp4_gap_1d}
\end{figure}

\begin{figure}[t]
  \centering
  \includegraphics[width=0.8\linewidth]{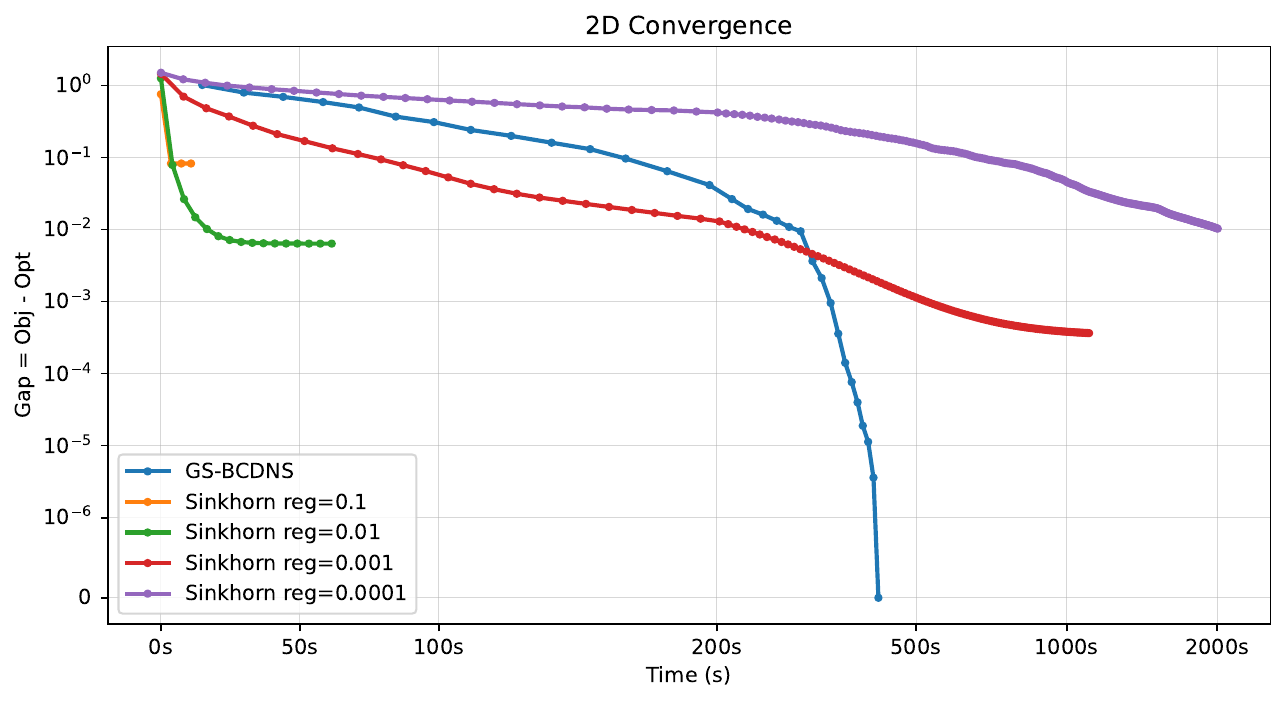}
  \caption{Test Problem 4 (2D): optimality gap versus time for GS-BCDNS and stabilized Sinkhorn under different
  regularization levels $\varepsilon$.}
  \label{fig:tp4_gap_2d}
\end{figure}

\subsection{Large-scale Experiment}\label{sec:large-scale}

This subsection evaluates the proposed method on a larger OT instance with
$n=4000$, where solving the full
problem directly becomes significantly more demanding.
We consider a one-dimensional Wasserstein distance problem between a uniform source
distribution and a heavy-tailed target distribution.

\paragraph{Test Problem 5 (Uniform--Beta).}
The source samples $\{u_i\}_{i=1}^n$ are drawn from $\mathrm{Unif}([-1,1])$.
The target samples $\{v_j\}_{j=1}^n$ are drawn from a $\mathrm{Beta}(0.5,0.5)$
distribution and then linearly mapped to $[-1,1]$.

To interpret the evolution of the computed transport plan, we use the
\emph{barycentric projection}
\[
T_k(u_i)
\;:=\;
\frac{1}{p_i}\sum_{j=1}^n v_j\,x^{k}_{ij},
\qquad i=1,\dots,n,
\]
where $x^k$ denotes the transport plan produced after $k$ iterations.
Intuitively, $T_k(u_i)$ is the average target location to which the mass in $u_i$
is transported under $x^k$.

\paragraph{Comment on Test Problem 5.}
Figure~\ref{fig:n4000_barycentric} summarizes the large-scale experiment.
The top panel plots the empirical marginals: the source is approximately uniform,
while the target (Beta$(0.5,0.5)$ mapped to $[-1,1]$) concentrates more mass near the
two endpoints, producing a characteristic U-shaped density.
This structural mismatch implies that the optimal coupling must transport a
nontrivial portion of the mass from the middle region of the uniform source toward
the two extremes of the target support.

The bottom panel shows the barycentric projection curves $T_k(x)$ at several
epochs ($k=10,20,50,100$) together with the final optimal map (``Optimal'').
At early epochs (e.g. $k=10$ and $k=20$), the barycentric projection exhibits
substantial oscillations, reflecting that the current sparse basic solution
still contains non-optimal variables and the coupling has not yet stabilized.
As the iterations proceed, the projection becomes increasingly structured and
approaches a smooth monotone map: by $k=50$ the curve already aligns closely
with the optimal trend, and by $k=100$ it nearly overlaps the optimal map over
most of the domain.

Consistent with this visual evidence, the proposed method reaches the globally
optimal solution at $218$ epochs in this experiment.
In general, Figure~\ref{fig:n4000_barycentric} demonstrates that even at
$n=4000$, the iterations produced by GS-BCDNS steadily reduce the optimality gap 
and the associated transport map progressively converges toward the exact
optimal coupling.

\begin{figure}[t]
  \centering
  \includegraphics[width=\linewidth]{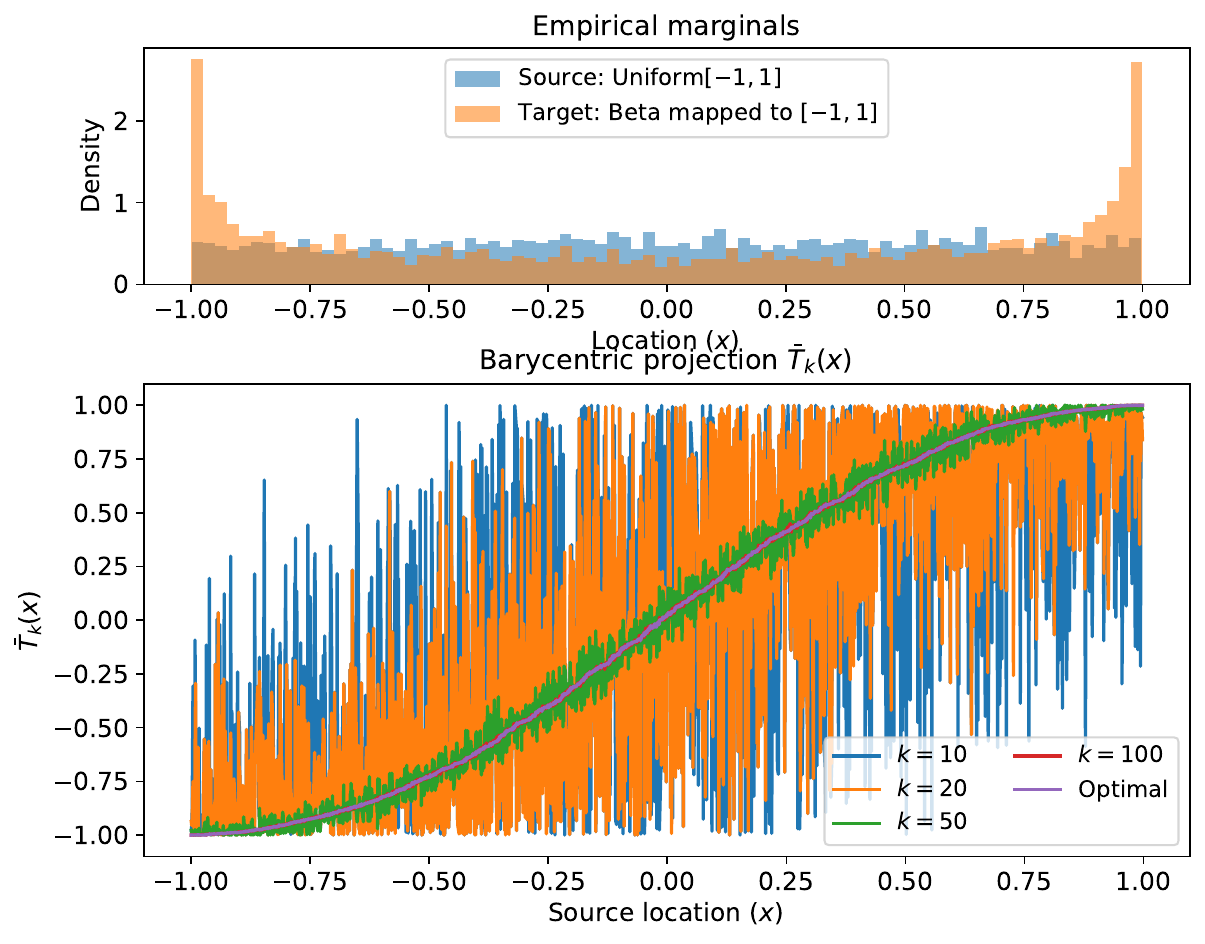}
  \caption{Large-scale OT experiment with $n=4000$.
  \textbf{Top:} empirical source and target marginals (Uniform$[-1,1]$ vs.\ Beta$(0.5,0.5)$ mapped to $[-1,1]$).
  \textbf{Bottom:} barycentric projection $T_k(x)$ at selected epochs $k$, illustrating progressive convergence of the
  transport map toward the optimal monotone coupling.}
  \label{fig:n4000_barycentric}
\end{figure}

\section{Conclusion}

In this paper, we studied efficient algorithms to solve the OT problem.
After reviewing representative approaches, including the NS method, BCD schemes, and the Sinkhorn algorithm, we proposed the BCDNS method.
The proposed framework integrates NS pivots with block coordinate updates and introduces the
\emph{basis variable succession} strategy to preserve feasibility across subproblems.
We established theoretical guarantees showing that the BCDNS method terminates in a finite number of iterations
and converges to an exact optimal solution.

Extensive numerical experiments demonstrate that the BCDNS method substantially improves computational efficiency
over the classical NS method on large-scale instances, achieving speed-ups of up to tens of times while
maintaining exact optimality.
Compared with the Sinkhorn algorithm, the BCDNS method exhibits a different accuracy--efficiency trade-off:
although it converges more slowly than Sinkhorn with large regularization parameters, it avoids
regularization bias and guarantees convergence to the OT optimum.
Moreover, when high accuracy is required, the BCDNS method can be competitive with or even faster than Sinkhorn.
These results indicate that the BCDNS method is particularly well suited for large-scale OT problems where exact solutions are required.

Several directions merit further investigation.
First, more sophisticated block selection strategies may further accelerate convergence.
Second, extending the proposed framework to more general network flow and structured linear programming
problems could broaden its applicability.


\printbibliography

@book{villani2009optimal,
  author    = {C. Villani},
  title     = {Optimal transport: old and new},
  year      = {2009},
  publisher = {Springer},
}

@book{vanderbei2001linear,
  author    = {R. J. Vanderbei},
  title     = {Linear programming: foundations and extensions},
  edition   = {2nd},
  publisher = {Springer Science+Business Media},
  year      = {2001},
}

@book{ahuja1993networkflows,
  author    = {Ravindra K. Ahuja and Thomas L. Magnanti and James B. Orlin},
  title     = {Network Flows: Theory, Algorithms, and Applications},
  publisher = {Prentice Hall},
  year      = {1993},
}

@article{grigoriadis1986efficient,
  author    = {M. D. Grigoriadis},
  title     = {An efficient implementation of the network simplex method},
  journal   = {Mathematical Programming Study},
  volume    = {26},
  pages     = {83--111},
  year      = {1986},
  publisher = {North-Holland},
}

@inproceedings{cuturi2013sinkhorn,
  author    = {M. Cuturi},
  title     = {Sinkhorn distances: lightspeed computation of optimal transport},
  booktitle = {Advances in Neural Information Processing Systems},
  volume    = {26},
  pages     = {2292--2300},
  year      = {2013},
}

@book{galichon2016optimal,
  author    = {A. Galichon},
  title     = {Optimal transport methods in economics},
  year      = {2016},
  publisher = {Princeton University Press},
}

@inproceedings{courty2016optimal,
  author    = {N. Courty and R. Flamary and A. Habrard and A. Rakotomamonjy},
  title     = {Optimal transport for domain adaptation},
  booktitle = {Advances in Neural Information Processing Systems},
  volume    = {29},
  pages     = {1--9},
  year      = {2016},
}

@article{schiebinger2019optimal,
  author    = {G. Schiebinger and J. Shu and M. Tabaka and et al.}, 
  title     = {Optimal-transport analysis of single-cell gene expression identifies developmental trajectories in reprogramming},
  journal   = {Cell},
  volume    = {176},
  number    = {6},
  pages     = {1517--1530.e19},
  year      = {2019},
}

@article{kantorovich1942translocation,
  author    = {L. V. Kantorovich},
  title     = {On the translocation of masses},
  journal   = {Doklady Akademii Nauk SSSR},
  volume    = {37},
  number    = {7–8},
  pages     = {199--201},
  year      = {1942},
}

@inproceedings{arjovsky2017wasserstein,
  title     = {Wasserstein generative adversarial networks},
  author    = {M. Arjovsky and S. Chintala and L. Bottou},
  booktitle = {Proceedings of the 34th International Conference on Machine Learning},
  volume    = {70},
  pages     = {214--223},
  year      = {2017},
  organization = {PMLR},
}

@article{monge1781deblais,
  author    = {G. Monge},
  title     = {Mémoire sur la théorie des déblais et des remblais},
  journal   = {Histoire de l'Académie Royale des Sciences},
  year      = {1781},
  pages     = {666--704},
}

@article{xie2023randomized,
  title     = {Randomized methods for computing optimal transport without regularization and their convergence analysis},
  author    = {Y. Xie and Z. Wang and Z. Zhang},
  journal   = {Journal of Scientific Computing},
  volume    = {100},
  number    = {37},
  year      = {2024},
}

@article{tseng2009block,
  title     = {Block-coordinate gradient descent method for linearly constrained nonsmooth separable optimization},
  author    = {P. Tseng and S. Yun},
  journal   = {Journal of Optimization Theory and Applications},
  volume    = {140},
  pages     = {513--535},
  year      = {2009},
}

@inproceedings{altschuler2017near,
  title     = {Near-linear time approximation algorithms for optimal transport via sinkhorn iteration},
  author    = {J. Altschuler and J. Niles-Weed and P. Rigollet},
  booktitle = {Advances in Neural Information Processing Systems (NeurIPS)},
  volume    = {30},
  year      = {2017},
}

@inproceedings{dvurechensky2018computational,
  title     = {Computational optimal transport: complexity by accelerated gradient descent is better than by sinkhorn’s algorithm},
  author    = {P. Dvurechensky and A. Gasnikov and A. Kroshnin},
  booktitle = {Proceedings of the International Conference on Machine Learning (ICML)},
  pages     = {1367--1376},
  year      = {2018},
  organization = {PMLR},
}

@article{kuhn1955hungarian,
  title     = {The Hungarian method for the assignment problem},
  author    = {H. W. Kuhn},
  journal   = {Naval Research Logistics Quarterly},
  volume    = {2},
  number    = {1-2},
  pages     = {83--97},
  year      = {1955},
  publisher = {Wiley Online Library},
}

@article{liao2022fast,
  title     = {Fast Sinkhorn I: Near-linear Complexity Optimal Transport for Seismic Waveform Inversion},
  author    = {Liao, Xiao and Zhang, Wei and Wang, Li and Sun, Jian},
  journal   = {Geophysical Journal International},
  volume    = {230},
  number    = {2},
  pages     = {1234--1248},
  year      = {2022},
  publisher = {Oxford University Press}
}

@article{bland1977new,
  author  = {R. G. Bland},
  title   = {New Finite Pivoting Rules for the Simplex Method},
  journal = {Mathematics of Operations Research},
  volume  = {2},
  number  = {2},
  pages   = {103--107},
  year    = {1977},
}

\end{document}